\documentclass{amsart}
\usepackage{amsmath, amsfonts, amsthm,amssymb,hyperref,graphicx,ulem,mathrsfs, ifpdf, color,verbatim }
\input xy
\xyoption{all}

\makeatletter
\newtheorem{theorem}{Theorem}[section]

\newtheorem{proposition}[theorem]{Proposition}

\theoremstyle{definition}

\theoremstyle{remark}

\numberwithin{equation}{section}

\begin{document}

\title{A Weak Fano Threefold Arising as a Blowup of a Curve of Genus 5 and Degree 8 on $\mathbb{P}^3$}

\author{Joseph W. Cutrone}
\author{Michael A. Limarzi}
\author{Nicholas A. Marshburn}

\begin{abstract} 
This article constructs a smooth weak Fano threefold of Picard number two with small anti-canonical morphism that arises as a blowup of a smooth curve of genus 5 and degree 8 in $\mathbb{P}^3$. While the existence of this weak Fano was known as a numerical possibility in \cite{CM13} and constructed in \cite{BL12}, this paper removes the dependencies on the results in \cite{JPR11} needed in the construction of \cite{BL12} and constructs the link in the style of \cite{ACM17}.
\end{abstract} 

\maketitle

\thispagestyle{empty}
\section{Introduction}
A smooth \textit{Fano} variety is a smooth projective variety whose anti-canonical class is ample. A smooth \textit{weak Fano} variety is a smooth projective variety whose anti-canonical class is both nef and big. These varieties with weakened positivity condition on the anti-canonical class arise in the study of birational maps between Fano varieties. The study and classification of smooth weak Fano threefolds with Picard number two began first with numerical classification.  Here numerical classification means the list of numerical possibilities for Sarkisov links between threefolds was provided.  In a series of two papers, by Jahnke, Peternell, and Radloff (see \cite{JPR05}, \cite{JPR11}), began the classification of smooth threefolds $X$ with big and nef (but not ample) anti-canonical divisor and Picard number two.  These papers considered different birational maps from smooth Fano threefolds $Y$ based on the classification type of the extremal contraction (as classified by Mori \cite{Mo82}) and the morphism associated with the base point free linear system $|-mK_X|$.  In particular, the authors classified, for both divisorial and small $|-mK_X|$ contractions, weak Fanos arising in links that were combination of extremal rays of type del Pezzo fibration, conic bundle, or birational contraction. Janke, Peternell, and Radloff provided numerical classifications as well as geometrical constructions when possible. Not all geometric constructions were provided and there are still some open cases. However, for $|-mK_X|$ small, the case when both contractions are birational was not considered.  

This final remaining combination of extremal contractions was first numerically classified by the authors in \cite{CM13}.  The tables in \cite{CM13} captured some geometric realizations already classically known, and the paper also provided other geometric constructions for numerical examples.  Later, further geometric proofs were provided for numerical cases, some previously open (e.g. \cite{ACM17}, \cite{BL12}, \cite{BL15}), to give explicit constructions. To date, the actual geometric existence of some cases in the tables of \cite{JPR05}, \cite{JPR11}, and \cite{CM13} are open problems.

This paper focuses on a particular Sarkisov link created when a smooth curve $C$ in $\mathbb{P}^3$ of degree 8 and genus 5 is blown-up.  We will show that, for a general curve $C$ not contained in a cubic surface, a symmetric link of type E1-E1 from $\mathbb{P}^3$ to $\mathbb{P}^3$ can be constructed.  It is important to note that this link is already known to exist both numerically and geometrically (see \cite{CM13}, \cite{BL12} respectively).  However, the construction in what follows does not rely upon the prior results and tables from the seminal paper \cite{JPR11}. Not only is this link interesting on its own right as an example of a Cremona map, (ex.,\cite{CS17}), but also the construction of the link in this paper removes any doubt regarding its geometric existence as there are small gaps in the tables of \cite{JPR11} which prior proofs rely on.  The authors hope that methods used will be useful to construct other open cases.

\section{Results}
In this paper, all varieties will be defined over the field of complex numbers, $\mathbb{C}$, and all varieties will be smooth unless stated otherwise.

The main result of this article is as follows:
 
\begin{theorem}
The numerical invariants listed in case 99 in \cite[Table E1-E1]{CM13} of the Sarkisov link starting with a smooth curve $C$ of degree 8 and genus 5 in $\mathbb{P}^3$ is geometrically realizable independent of the results in \cite{JPR11}.
\end{theorem}

In particular, there are two places where the dependence of the tables of \cite{JPR11} needs to be removed.  Firstly, the claim that the anti-canonical morphism is in fact small is obtained from \cite{JPR11}.  We will show this directly in \ref{kxissmall}. The second time the tables in \cite{JPR11} are referenced in prior papers is to show that the extremal contraction after the flop is of divisorial type.  This claim is shown explicitly in Proposition \ref{linkise1e1}.  In addition to removing dependence on the tables in \cite{JPR11}, explict details are worked out for the geometric construction of this particular Sarkisov link arising as the blowup of a smooth curve $C$ of degree 8 and genus 5 in $\mathbb{P}^3$.\\

We begin with some preliminaries. 

\section{Preliminaries} \label{Prelim}
Let $X$ be a weak Fano threefold of Picard number two such that the anti-canonical system $|{-}K_X|$ is free and gives a small contraction $\psi\colon X \rightarrow X^\prime$. By \cite{Ko89}, the $K_X$-trivial curves can be flopped. More precisely, there is a commutative diagram
~\\
$$\xymatrix{
X \ar@{-->}[rr]^-\chi \ar[d]_-{\phi} \ar[rd]^-\psi & & X^+ \ar[ld]_-{\psi^+} \ar[d]^-{\phi^+} \\
Y & X^\prime & Y^+,
}$$
~\\
where $\chi$ is an isomorphism outside of the exceptional locus of $\psi$. The strict transform of a divisor $D \in \text{Pic(X)}$ across the flop $\chi$ is denoted by $\widetilde{D}$.  Since $\chi$ is small, $\widetilde{K_X} = K_{X^+}.$  

The particular case of study in what follows is case No. 99 of \cite{CM13} of type E1-E1.  Recall that type E1-E1 means that both $\phi$ and $\phi^+$ are assumed to be divisorial contractions of type E1 in the sense of \cite{Mo82}. In particular, $\phi$ is the blow-up of a smooth irreducible curve $C$ of degree 8 and genus 5 in $Y = \mathbb{P}^3$. Likewise $\phi^+$ is the blow-up of $Y^+$ along a smooth irreducible curve $C^+$ of degree 5 and genus 8 in $\mathbb{P}^3$.  The exceptional divisors of the blow-ups $\phi$ and $\phi^+$ are denoted by $E$ and $E^+$, respectively.

\section{Construction} \label{Construction}
To begin the construction of the link, let us first guarantee the existence of such a curve, and in particular, choose the right curve needed to construct the link.  
 
By \cite{Knu02}, Theorem 1.1 on pg 202, we can find a smooth irreducible curve $C$ of genus 5 and degree 8 lying on a smooth K3 quartic surface $S\subset \mathbb{P}^3$ with the property that Pic($S) = \mathbb{Z}H_S \oplus \mathbb{Z}C$. 

\begin{proposition} \label{pro1}Let $X$ be the blow-up of a smooth irreducible curve $C$ of genus 5 and degree 8 lying on a smooth K3 quartic surface $S\subset \mathbb{P}^3$ with the property that Pic($S) = \mathbb{Z}H_S \oplus \mathbb{Z}C$. Then $X$ is weak Fano. 
\end{proposition}

\begin{proof}
Let $X$ be the blow-up of $C(5,8) \subset S \subset \mathbb{P}^3$.  The anticanonical divisor of $\mathbb{P}^3$, $-K_{\mathbb{P}^3}$, is $4H$, so the anticanonical divisor of $X$, $-K_X$, is then $4H-E$, where $E$ is the exceptional divisor.  The freeness of the linear system $|-K_X| = |4H-E|$ is studied by considering the freeness of the linear system on $\mathbb{P}^3$ of $|4H-C|$ outside of $C$. This linear system is studied by considering the linear system $|4H_S - C|$ on the smooth K3 quartic surface $S$.  

There are two observations to make regarding the linear system $|4H_S - C|$ on $S$.  The first is that this linear system is free (and therefore nef) by the following proposition with $n = 2$, $d = 8$, $g = 5$, and $k = 4$. For the proof, see Proposition 2.1 in \cite{ACM17}.

\begin{proposition} \label{pro2}
Let $S$ be a smooth K3 with $\mathrm{Pic}(S) =\mathbb{Z}H \oplus \mathbb{Z}C$, with $H$ very
ample and $C$ a smooth (irreducible) curve. Assume $H^2=2n$, $C \cdot H=d$ and
$C^2=2(g-1)$. Let $k>0$ be an integer. Then $kH-C$ is nef if and only if
\vspace{0.2cm}

$2nk>d, nk^2-dk+g-1 \geq 0$ and $(2nk-d,nk^2-dk+g) \neq (2n+1,n+1)$.
\vspace{0.2cm}

\noindent Furthermore, $kH-C$ is free if and only if it is nef and we
are \textbf{not} in the case
\vspace{0.2cm}

$d^2-4n(g-1)=1$, and $2nk-d-1$ or $2nk-d+1$ divides $2n$.
\vspace{0.2cm}
\end{proposition}

The second observation to make is that, with the exception of $S$ itself, every divisor in $|4H-C|$ on $\mathbb{P}^3$ is the pullback of a divisor on the quartic surface $S$ since there is a surjection from $H^0(4H-C)$ to $H^0(4H_S- C)$.  To prove this claim, consider the short exact sequence $$0 \rightarrow \mathcal{O}_{\mathbb{P}^3}(-4) \hookrightarrow \mathcal{O}_{\mathbb{P}^3}\twoheadrightarrow\mathcal{O_S} \rightarrow 0.$$  After twisting by 4, the short exact sequence becomes $$0 \rightarrow \mathcal{O}_{\mathbb{P}^3} \hookrightarrow \mathcal{O}_{\mathbb{P}^3}(4) \twoheadrightarrow \mathcal{O_S}(4) \rightarrow 0.$$  The long exact sequence of cohomology is $$0 \rightarrow H^0(\mathcal{O}_{\mathbb{P}^3}) \rightarrow H^0(\mathcal{O}_{\mathbb{P}^3}(4)) \rightarrow H^0(\mathcal{O}_S(4))\rightarrow H^1(\mathcal{O}_{\mathbb{P}^3})\rightarrow \ldots$$ and since $H^1(\mathcal{O}_{\mathbb{P}^3})=0$, there is a surjection $H^0(\mathcal{O}_{\mathbb{P}^3}(4)) \twoheadrightarrow H^0(\mathcal{O}_S(4))$.  Then any member of the complete linear system $|4H_S|$ is a restriction of a member of $|4H|$.
Since $-K_X$ is free, $-K_X$ is nef.

To show $-K_X$ is big, it suffices to check that $({-}K_X)^3>0$, \cite[Thm.2.2.16, p.144]{LazI}. Using the formula 
$$({-}K_X)^3=({-}K_Y)^3+2K_Y\cdot C -2+2g = 64-8d-2+2g,$$ with $d = 8$ and $g = 5$ gives $-K_X^3 = 8$.

To show that $-K_X$ is not ample, it suffices to show $C$ has at least one 4-secant line.  In fact, $C$ has 10 4-secant lines.  This can be calculated directly using \cite{LeB82} with $d = 8$ and $g = 5$. $$\text{Number of 4-secant lines to } C = \frac{(d-2)(d-3)^2(d-4)}{12} - \frac{(d^2-7d+13-g)g}{2}$$
\end{proof}

To show that the anticanonical morphism only contracts these curves and not a divisor (that is, the anticanonical morphism $\psi: X \to X'$ is a small contraction), the traditional approach has been to compare the invariants with the classification tables in \cite{JPR11}.  Since this case of blowing up a smooth curve $C$ of genus $g = 5$ and degree $d = 8$ is not on their tables when $C$ is not contained in a cubic surface, the small contraction induces a flop $\chi:X \dashrightarrow X^+$ by \cite{Ko89}.  The contraction on the other side of the Sarkisov link must then be of type $E1 - E^*$, where * = 1,2,3,4,5.  Using the tables in \cite{CM13}, one then checks the weak Fano constructed from blowing up a smooth curve $C$ of genus $g = 5$ and degree $d = 8$ is not on the E1-E2,E1-E3/E4,E1-E5 tables to conclude the link must be of type E1-E1.  

While the traditional method is effective, and follow-up papers (e.g., \cite{BL12}, \cite{ACM17}), have verified the numerical possibilities with explicit geometric constructions, there are known gaps in the tables in \cite{JPR11}.  It is the authors' understanding these gaps will be resolved in an upcoming paper, but until then, the main result of this paper is to construct the Sarkisov link from the blow-up of the $C$(5,8) curve in $\mathbb{P}^3$ alone, not relying on any prior classifications.  

Until now, we have only the following diagram: 
$$\xymatrix{
X \ar[d]_-{\phi} \ar[rd]^-\psi 	& 				 \\
\mathbb{P}^3 																				& X^\prime
}$$

where $\phi$ is the blow up of a degree 8 and genus 5 curve $C$ in $\mathbb{P}^3$. We will show that $\psi$ must be a small contraction.

\begin{proposition}\label{kxissmall}  Let $X$ and $C$ be as in \ref{pro1}. Then $\psi$ as in the diagram above must be a small contraction. 
\end{proposition}

\begin{proof}Let $D$ be a divisor on $X$. Assume $\psi(D) = B$, where $B$ is a curve.  Since $D$ is contracted, $(-K_X)^2D = 0$.  Writing $D = aH + bE$ for some integers $a,b$ and writing $-K_X$ as $4H - E$, the equation $(-K_X)^2D = 0$ becomes:

\begin{align*}
0 &= (4H - E)^2(aH+bE) \\
	&= 16aH^3 + (16b-8a)H^2E + (a-8b)HE^2 + bE^3\\
	&= 16a + 0 - 8(a-8b) + b(-(4)(8) - (2)(5) + 2).\\
\end{align*}

Simplifying gives $a = -3b$, so $D = k(3H-E)$, for some integer $k$. Note that if $k =1$, then $C$ is contained in a cubic surface $S$ in $\mathbb{P}^3$.  The surface $S$ is in $\mathbb{P}^3$, so first identify $D$ with $k(3H-C)$ in $\mathbb{P}^3$.  Consider the restriction of $D$ to $S$ and consider the divisor $D|_S - C$ on $S$.  Then $k(3H-C) \sim C^\prime$ on $S$ for some effective divisor $C'$ with degree $4b$.  Then
$$\begin{array}{cl}
(C^\prime)^2 & = (3kH-kC)^2\\
						& = 9k^2H^2 - 6k^2HC + k^2C^2\\
						& = k^2(36-48+2(5)-2)\\
						& = -4k^2
\end{array}
$$

By \cite{S-D74}, $(C')^2 \le -4$ means $C'$ decomposes into a sum of irreducible curves of which at least two are rational. Since a rational curve on a K3 has self-intersection of -2, $S$ contains no rational curves since $H$ and $C$ generate Pic($S)$ and since $H^2$, $H.C$, and $C^2$ are all in $4\mathbb{Z}$.  
\end{proof}

By \cite{Ko89}, $\psi$ induces a flop which is an isomorphism outside of the exceptional locus to a smooth weak Fano threefold $X^+$ with Picard number 2 and we have the following diagram:
\label{link}$$\xymatrix{X \ar@{-->}[rr]^{\chi} \ar[d]^{\phi} \ar[dr]^{\psi}& & X^{+} \ar[d]^{\phi^{+}} \ar[dl]_{\psi^{+}} \\
          \mathbb{P}^3 & X' & Y^{+}}$$

These 10 4-secants are indeed the 10 flopping curves as indicated by the defect $e$ of the flop $\chi$.  Recall the defect of a flop is defined as $$e = E^3 - \widetilde{E}^3$$ where 
\label{ecubed} $$\widetilde{E}^3 = (\alpha^+)^3(-K_X)^3 + 3(\alpha^+)^2\beta^+\sigma^+-3\alpha^+(\beta^+)^2K_X(E^+)^2 + (\beta^+)^3(E^+)^3$$ and $$\widetilde{E^+}^3 = \alpha^3(-K_X)^3 + 3 \alpha^2\beta\sigma - 3 \alpha\beta^2K_XE^2 + \beta^3E^3$$ in the notation of \cite{CM13}.  In particular, $e/r^3 = 10$, where $r$ is the index of $\mathbb{P}^3$ (in this case, $r = 4$).

Once again, the traditional approach to classifying this link would then be to look at the tables in \cite{JPR11} to see if this link showed up as a del Pezzo or conic bundle contraction. If not, the link was then of the form E1-E* (*=1,2,3,4,5), and the results of \cite{CM13} were used.  Instead of relying on the tables, we will show directly that $\phi^+$ cannot be of type del Pezzo or conic bundle fibration.   

\begin{proposition} \label{linkise1e1} The morphism $\phi^+$ as in the above diagram is of exceptional type.  That is, it is an exceptional contraction of type E. In particular, it is case 99 on the E1-E1 table of \cite{CM13}
\end{proposition}

\begin{proof} Since the Sarkisov link exists, by the classification of Mori, the extremal contraction is either a divisorial contraction (of type E1,E2,E3,E4,E5), or a relative Fano model to a surface (a conic bundle contraction of type C1 or C2), or to a curve (a del Dezzo contraction of type D1, D2, or D3.).  Let us show that neither the conic bundle contraction nor the del Pezzo contraction can occur, forcing $\phi^+: X^+ \to Y^+$ to be a divisorial contraction.  

Assume $\phi^+:X^+ \to Y^+$ is a conic bundle and let $D$ be the pullback of a line in $\mathbb{P}^2.$  Then $-K_{X^+}D^2 = 2.$  Compute this intersection number across the flop, writing $\widetilde{D} = a(-K_X) + bE$ for some rational $a,b$.  Then, using the fact that $$(-K_X)^2E = rd+2-2g = (4)(8) + 2-2(5)=24$$ and $$-K_XE^2 = 2g-2 = 2(5)-2 = 8,$$
\begin{align*}
2 & = -K_{X^+}D^2\\
	& = -K_X\widetilde{D} \\
	& = -K_X(a(-K_X) + bE)^2\\
	& = a^2(-K_X)^3 + 2ab(-K_X^2E) + b^2(-K_XE^2)\\
	& = 8a^2 + 48ab + 8b^2.\\
\end{align*}

Since 8 divides the right side but not the left, there are no solutions. 

A similar argument shows $\phi^+: X^+ \to Y^+$ is not a del Pezzo fibration. Suppose $\phi^+: X^+ \to \mathbb{P}^1$ is a del Pezzo fibration.  Let $D$ be the divisor class of the del Pezzo surfaces in the fibration. Then $-K_X^+D^2 = 0$.  Using $\widetilde{D} = a(-K_X) + bE)$ for some rational $a,b$, by the above calculations, $0 = 8a^2 + 48ab + 8b^2$. 

Setting $b = 1$ then gives the quadratic $a^2 + 6a + 1 = 0$, which has no rational solutions for $a$, so $\phi^+: X^+ \to Y^+$ is not a del Pezzo fibration.  Similarly, $b = 0$ gives no solutions either. 

By classification of extremal contractions, $\phi^+: X^+ \to Y^+$ must be of divisorial type.  In particular, using the tables in \cite{CM13}, the link is of type E1-E1, with numerical constraints of case 99.  
\end{proof}

This proposition now places $\mathbb{P}^3$ as $Y^+$ in the E1-E1 link and a geometric construction satisfying the numerical constraints of case 99 in \cite{CM13} is complete.  In particular, $\phi^+$ must be the blowdown of the exceptional divisor to a curve $C^+$ in $Y^+ = \mathbb{P}^3$ of genus 5 and degree 8.

\section{Conclusion}
The above calculations construct the Sarkisov link with small anti-canonical morphism arising as the blow up of a general smooth curve $C$ of genus 5 and degree 8 in $\mathbb{P}^3$, where $C$ is not contained in a cubic surface, without any reference to the tables in \cite{JPR05} or \cite{JPR11}.  In particular, the proof that the anti-canonical morphism is small and the fact that the link contains two exceptional divisorial contractions was shown without reference to prior classification tables.  

This link satisfies the numerical constraints in Table E1-E1 in \cite{CM13}, case number 99.  The link consists of 10 flopping curves, each of which is a 4-secant to the original curve $C$, and is symmetric in that the morphism $\phi^+:X^+ \to Y^+ = \mathbb{P}^3$ contracts the exceptional divisor to a smooth curve $C$ of genus 5 and degree 8.  

The final diagram of the E1-E1 Sarkisov link is as follows:

$$
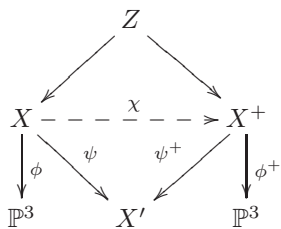
\xymatrix{& Z \ar[rd] \ar[ld] & \\  X \ar@{-->}[rr]^{\chi} \ar[d]^{\phi} \ar[dr]^{\psi}& & X^{+} \ar[d]^{\phi^{+}} \ar[dl]_{\psi^{+}} \\
          \mathbb{P}^3 & X' & \mathbb{P}^3}$$

It is interesting to note that similar arguments may be able to be applied to case number 76 on the E1-E1 table in \cite{CM13}.  This case, which was already shown to geometrically exist in \cite{JPR05}, consists of the symmetric Sarkisov link from $\mathbb{P}^3$ to $\mathbb{P}^3$, blowing up a smooth curve $C$ of genus 11 and degree 10.  We hope to investigate this case, as well as the other known open cases of geometric existence, using the arguments developed in this paper in the future. 

\vspace{1cm}

\section*{Acknowledgments} The authors express their gratitude to I. Cheltsov for both posing the problem and subsequent discussions that followed. In addition, we would like to thank all the organizers of the Workshop in Algebraic Geometry held on December 18-22, 2016, in Hanga Roa, Chile.


\newpage

\begin{thebibliography}{99}
\bibitem[ACM17]{ACM17} Arap, M., Cutrone, J., Marshburn, N.: \textit{On the existence of cetrain weak fano threefolds of picard number two}, Mathematica Scandinavica,[S.I.], v.120,n.1,p68-86, Feb. 2017 ISSN 1903-1807.
\bibitem[AM14]{AM14} Arap, M., Marshburn, N.: \textit{Brill-Noether general curves on Knutsen K3 surfaces.} To appear in Comptes rendus - Math\'ematique (2014). 
\bibitem[BL12]{BL12} Blanc, J., Lamy, S.: \textit{Weak Fano threefolds obtained by blowing-up a space curve and construction of Sarkisov links}. Proc. Lond. Math. Soc. (3) \textbf{105} (2012), no. 5, 1047--1075. 
\bibitem[BL15]{BL15} Blanc, J., Lamy, S.: \textit{On Birational Maps From Cubic Threefolds}. North-West. Eur. J. Math. 1 (2015), 55-84. 
\bibitem[CS17]{CS17} Cheltsov, V., Shramov, C.: \textit{Finite Collineation Groups and Birational Rigidity} (to appear.)
\bibitem[CM13]{CM13} Cutrone, J.W., Marshburn, N.A.: \textit{Towards the classification of weak Fano threefolds with $\rho = 2$.} Cent. Eur. J. Math. \textbf{11} (2013), no. 9, 1552--1576.
\bibitem[Isk89]{Isk89} Iskovskikh, V. A.: \textit{Double projection from a line onto Fano $3$-folds of the first kind.} (Russian) Mat. Sb. \textbf{180} (1989), no. 2, 260--278; translation in Math. USSR-Sb. \textbf{66} (1990), no. 1, 265--284.
\bibitem[Isk77]{Isk77}Iskovskih, V. A.: \textit{Fano threefolds. I.} (Russian) Izv. Akad. Nauk SSSR Ser. Mat. \textbf{41}, no. 3 (1977),  516--562. 
\bibitem[Isk78]{Isk78}Iskovskih, V. A.: \textit{Fano threefolds. II.} (Russian) Izv. Akad. Nauk SSSR Ser. Mat. \textbf{42}, no. 3 (1978), 506--549. 
\bibitem[IP99]{IP99} Iskovskikh, V. A., Prokhorov, Yu. G.: \textit{Fano varieties. Algebraic
geometry V}, Encyclopaedia Math. Sci., \textbf{47}, Springer, Berlin, 1999.
\bibitem[JPR05]{JPR05} Jahnke, P., Peternell, T., Radloff, I.: \textit{Threefolds with big
and nef anti-canonical bundles I.} Math. Ann. \textbf{333} (2005), no. 3, 569--631.
\bibitem[JPR11]{JPR11} Jahnke, P., Peternell, T., Radloff, I.: \textit{Threefolds with big
and nef anti-canonical bundles II.} Cent. Eur. J. Math. \textbf{9} (2011), no. 3, 449--488.
\bibitem[KL72]{KL72} Kleiman, S. L., Laksov, D.: \textit{Schubert calculus.} Amer. Math. Monthly \textbf{79} (1972), 1061--1082.
\bibitem[Knu02]{Knu02} Knutsen, A.L.: \textit{Smooth curves on projective $K3$ surfaces.} Math. Scand. \textbf{90} (2002) 215--231.
\bibitem[Knu13]{Knu13} A.L. Knutsen: \textit{Smooth, isolated curves in families of Calabi-Yau threefolds in homogeneous spaces.} J. Korean Math. Soc. \textbf{50} (2013), No. 5, pp. 1033--1050.
\bibitem[Ko89]{Ko89} Koll\'ar, J.: \textit{Flops.} Nagoya Math. J. \textbf{113} (1989), 15--36.
\bibitem[LazI]{LazI} Lazarsfeld, R.: \textit{Positivity in algebraic geometry I. Classical
setting: line bundles and linear series.} Ergeb. der Math. und ihrer Grenz. 3. Folge, \textbf{48}. Springer-Verlag, Berlin, 2004.
\bibitem[LeB82]{LeB82} Le Barz, P.: \textit{Formules multis\'ecantes pour les courbes gauches
quelconques.} (Nice, 1981), pp. 165--197, Progr. Math., \textbf{24}, BirkhŠuser, Boston, Mass., 1982.
\bibitem[Ma73]{Ma73} Maruyama, M.: \textit{On a family of algebraic vector bundles.} Number theory, algebraic geometry and commutative algebra, in honor of Yasuo Akizuki, pp. 95Ð146. Kinokuniya, Tokyo, 1973.
\bibitem[Mo82]{Mo82} Mori, S.: \textit{Threefolds whose canonical bundles are not numericallyveffective.} Ann. of Math. (2) \textbf{116} (1982), no. 1, 133--176.
\bibitem[MM84]{MM84} Mori, S., Mukai, S.: \textit{Classification of Fano 3-folds with $B_2 \ge 2$. I.} Algebraic and topological theories (Kinosaki, 1984), 496--545, Kinokuniya, Tokyo, 1986.
\bibitem[Mu93]{Mu93} Mukai, S.: \textit{Curves and Grassmannians.} Algebraic geometry and related topics (Inchon, 1992), Int. Press, Cambridge, MA (1993) 19--40.
\bibitem[Mu02]{Mu02}  Mukai, S.: \textit{New development of theory of Fano 3-folds: vector bundle method and moduli problem.} Sugaku \textbf{47} (1995), no. 2, 125--144; translation in: Sugaku Expositions \textbf{15} (2002), no. 2, 125--150.
\bibitem[Mu10]{Mu10} Mukai, S.: \textit{Curves and symmetric spaces, II.} Annals of Math., \textbf{172} (2010), 1539--1558.
\bibitem[S-D74]{S-D74} Saint-Donat, B.: \textit{Projective models of K3 surfaces.} Amer. J. Math. \textbf{96} (1974), 602--639.
\bibitem[Sh79]{Sh79} Shokurov, V. V.: \textit{The existence of a line on Fano varieties.} (Russian) Izv. Akad. Nauk SSSR Ser. Mat. \textbf{43} , no. 4 (1979), 922--964
\bibitem[Tak09]{Tak09} Takeuchi, K.:  \textit{Weak Fano threefolds with del Pezzo fibration.} arXiv preprint (2009).
\end{thebibliography}
\end{document}